\documentclass[12pt,reqno]{amsart}
\usepackage{amssymb,latexsym,amsmath,epsfig,amsthm,mathrsfs}

\usepackage{rotating}
\usepackage{graphicx}
\usepackage{amssymb}
\usepackage{lineno}
\usepackage{enumitem}
\usepackage{cite}
\usepackage{multicol}
\usepackage[top=3cm, bottom=3cm, left=2.5cm, right=2.5cm]{geometry}
\usepackage[usenames]{color}
\usepackage[colorlinks=true,
linkcolor=blue,
filecolor=blue,
citecolor=blue]{hyperref}

\numberwithin{equation}{section}

\newtheorem{question}{Question}[section]
\newtheorem{proposition}{Proposition}[section]
\newtheorem{corollary}{Corollary}[section]
\newtheorem{definition}{Definition}[section]

\newtheorem{theorem}{Theorem}[section]
\newtheorem{lemma}[theorem]{Lemma}

\begin{document}
	
	\title[On  additive complements in the complement of a set of natural numbers]{On  additive complements in the complement of a set of natural numbers}
	
	\author[B R Patil]{Bhuwanesh Rao Patil}
	\address{Department of Applied Science and Humanities, Government Engineering College, Nawada, Bihar, 805111, India}
	\email{bhuwanesh1989@gmail.com}
	

	
	\author[Mohan]{Mohan$^{\ast}$}
	\address{Stat-Math Unit, Indian  Statistical Institute,  S. J. S. Sansanwal Marg,  New Delhi, 110016, India}
	\email{mohan98math@gmail.com}
	\thanks{$^{\ast}$The corresponding author}
	
	
	
	

	\begin{abstract}
		Let $A$  be a set of natural numbers.	A set $B$, a set of natural numbers, is an additive complement of the set $A$  if all sufficiently large natural numbers can be represented in the form $x+y$, where $x\in A$ and $y\in B$. Erd\H{o}s proposed a conjecture that every infinite set of natural numbers has a sparse additive complement, and in 1954, Lorentz proved this conjecture.  This article describes the existence or non-existence of those additive complements of the set $A$  that is a subset of the complement of $A$. We provide a ratio test to verify the existence of such additive complements. In precise, we prove that if   $A=\{a_i: i\in \mathbb{N}\}$ is a set of natural numbers such that  $a_i<a_{i+1} $ for $i \in \mathbb{N}$ and $\liminf_{n\rightarrow \infty } (a_{n+1}/a_{n})>1$, then there exists a set $B\subset \mathbb{N}\setminus A$  such that $B$ is a sparse additive complement of the set $A$.

	\end{abstract}
	
	\maketitle
	
	
	\section{Notations}
	Let $\mathbb{N}$ be  the set of all natural numbers and $\mathbb{Z}$ be the set of all  integers. The set $\mathcal{P}(\mathbb{N})$ is the collection of all subsets of natural numbers. For a set $A$, $|A|$ denotes the cardinality of $A$. For any real valued function $f$,  we assume $\sum_{s\in S}f(s) = 0$ and $\bigcup_{s\in S}f(s) = \varnothing $ if $S = \varnothing$. If $A$ is a subset of some set $S$, then the indicator function of the set $A$, denoted by $1_{A} $, is defined from $S$ to $\{0,1\}$ by 
	\begin{equation*}
		1_{A}(x) =
		\begin{cases}
			1, & x\in A;\\
			0, & x\notin A.
		\end{cases} 
	\end{equation*}
	By $\max(A)$, we mean the largest element of the set $A$. For a given real number $x$, $\lfloor x \rfloor$ denotes the largest integer less than or equal to $x$.	For given real numbers $a$ and $b$, we denote
	$$(a,b):=\{x\in \mathbb{Z}: a<x<b \}, \quad
	(a,b]:=\{x\in \mathbb{Z}: a<x\leq b \},$$
	$$[a,b):=\{x\in \mathbb{Z}: a\leq x<b \}, \quad
	[a,b]:=\{x\in \mathbb{Z}: a\leq x\leq b \}.$$
	For given  subsets $A$ and $B$ of $\mathbb{N}$ and an integer $u$, we define $$B+A = A+B:=\{x+y:x\in A,y\in B\}, \quad \text{} B+u = u+B:=\{u+y:y\in B\},$$
	$$ \text{and} ~ u-B:=\{u-y:y\in B\}.$$
	For given sets $X$ and $Y$, $X\setminus Y := \{x\in X: x \notin Y\}$. The set $X \setminus Y$, also known as the complement of $Y$ in $X$. For any subset $A$ of $\mathbb{N}$, ``the complement of  set  $A$" means $\mathbb{N}\setminus A$ throughout the article.

	\section{Introduction}
	A simple Diophantine equation $x+y = n$ always has a solution among nonnegative integers for every natural number $n$. However, finding solutions will not be so easy if we choose   $x$ and $y$ from the prescribed set of natural numbers. This motivates us to study a solution of the Diophantine equation $x+y=n$, $x\in A$ and $y\in B$, where $A$ and $B$ are sets of non-negative integers for all sufficiently large natural numbers $n$. In general,  we are interested in studying those pairs $(A, B)$ for which $\mathbb{N}\setminus (A+B)$ is a finite set. This sparks curiosity on the following question:
	\begin{question}
		Let $A$ be a non-empty set of natural numbers. Does there exists $B$, a set of natural numbers, such that $\mathbb{N}\setminus (A+B)$ is a finite set. 
	\end{question}
	\noindent This leads to the following definition.
	\begin{definition}
		Let $A$ be a set of natural numbers. Then, a set $B$, a set of natural numbers, is an additive complement of $A$ if the complement of $A+B$ in the set of all natural numbers is a finite set.
	\end{definition} 
	\noindent
	Note that the set of all natural numbers is always an additive complement to every non-empty set of natural numbers. So, one can be curious to look for other special additive complements with certain structures. Examples of such additive complements can be sparse additive complements, additive complements of the form of the union of disjoint infinite arithmetic progressions, additive complements of a set inside the complement of that set, etc. By considering a set with zero density  (see Definition \ref{density_definition}) as a sparse set,
	Erd\H{o}s proposed a conjecture that every infinite set of natural numbers has a sparse additive complement, and in 1954, Lorentz proved this conjecture. Additive complements of the form of the union of disjoint infinite arithmetic progressions are described by Mohan et al. \cite{Mohan_Patil_2024}. One may see  \cite{Lorentz, Mohan_Patil_2024, erdos, Grekos,Faisant, Leonetti1,Leonetti2} and references therein for similar types of results. In this article, we shall study additive complements of a set in the complement of that set. Such type of additive complements is defined below: 
	\begin{definition}
		Let $A$ be a set of natural numbers. Then $B$, a set of natural numbers, is an additive complement of $A$ in the complement of $A$ if $B$ is an additive complement of $A$ and $B\cap A=\varnothing$.
	\end{definition} 
	\noindent For example, the set of all primes acts as a sparse additive complement of the set of all composite natural numbers. However, the existence of an additive complement of a set in the complement of that set is not necessary; for example, there is no additive complement of the set of all even natural numbers in the set of odd natural numbers. In fact, by proving the following result, Mohan et al. \cite{Mohan_Patil_2024} gave an infinite family of sets $A$ such that every additive complement of the set $A$ intersects the set $A$.
	\begin{proposition}\textup{\cite[Mohan, Patil and Pandey]{Mohan_Patil_2024}}\label{Thm-1} Let $\alpha$ be a real number such that $0 \leq \alpha \leq 1$. Let $S$ be an infinite set of  natural numbers. Then
		\begin{enumerate}
			\item [\upshape(1)] there exists a set $A\subseteq \mathbb{N}$  such that $(A+(\mathbb{N}\setminus A))\cap S=\varnothing,$ 
			\item [\upshape(2)] there exists a set  $A\subseteq \mathbb{N}$ with $d(A)=\alpha$ such that there is no additive complement of $A$ in the complement of $A$.
		\end{enumerate} 
	\end{proposition}
	\noindent So, one can be interested in the following question:
	\begin{question}
		Characterize all set $A$ of natural numbers such that there exists an additive complement of the set $A$ in the complement of the set $A$. 
	\end{question}
	\noindent Now, we answer the above question in the form of the following theorem by providing a family of sparse sets of natural numbers.
	\begin{theorem}[\upshape{Ratio test for the existence of additive complement in the complement}]\label{Ratio test}
		Let $A = \{a_{i} \in \mathbb{N} : i\in \mathbb{N}\}$ be an infinite set of natural numbers such that $a_{i}<a_{i+1}$ for all $i \in \mathbb{N}$ and $\liminf_{n\rightarrow \infty }(a_{n+1}/a_{n}) >1$, then there exists a set $B$, a set of natural numbers with density zero, such that $B$ is an additive complement of $A$ in the complement of $A$.
	\end{theorem}
	\noindent Proof of the above theorem given in Section \ref{Proof of Theorem}.
	
	The conclusion of the above theorem may not hold if  $\liminf_{n\rightarrow \infty } (a_{n+1}/a_{n}) =1$. Using Proposition \ref{Thm-1} and Theorem \ref{Ratio test}, we can construct examples of such sets $A$. Moreover, we have the following corollary.
	\begin{corollary}\label{non-exitance of additive complement}
		Let $\alpha$ be a real number such that $0 \leq \alpha \leq 1$. Then there exists  a set $A = \{a_{i} \in \mathbb{N}: i\in \mathbb{N}\}$ with $a_{i}<a_{i+1}$, $d(A) = \alpha$ and $\liminf_{n\rightarrow \infty } (a_{n+1}/a_{n}) =1$ such that there is no additive complement in the complement of $A$.
	\end{corollary}
	
	\section{Preliminary}
	
	\begin{definition}\label{density_definition}
		Let $A$ be a set of natural numbers. Then the upper asymptotic density and lower asymptotic density, denoted by $\overline{d}(A)$ and $\underline{d}(A)$ respectively, is defined as follows
		\[\overline{d}(A) = \limsup_{n\rightarrow \infty} \dfrac{|A \cap [1,n]|}{n} \quad \text{and} \qquad \underline{d}(A) = \liminf_{n\rightarrow \infty} \dfrac{|A \cap [1,n]|}{n}. \]
		If $\overline{d}(A) = \underline{d}(A) $, then we say that density of $A$, denoted by $d(A)$, exists  where $d(A)=\overline{d}(A) = \underline{d}(A)$. 
	\end{definition}
	
	\begin{lemma}\label{density_calculation}
		Let $(x_i)_{i=1}^{\infty}$ be an increasing sequence of natural numbers. Then 
		$$\liminf_{n\rightarrow \infty }\dfrac{1}{n}\sum_{i=1}^{n}\dfrac{\log x_i}{x_i}=0.$$ 
	\end{lemma}
	\begin{proof}
		Let $\epsilon>0$ be a real number. Since $\lim_{n\rightarrow \infty}\dfrac{\log n}{n}=0$ by L'Hopital rule, there exists $n_0\in \mathbb{N}$ such that
		$$\dfrac{\log x}{x}<\epsilon ~\forall ~ x\geq n_0.$$
		Therefore, for each $t> n_0$, we have
		$$\dfrac{1}{t}\sum_{i=1}^{t}\dfrac{\log x_i}{x_i}\leq\dfrac{1}{t} \sum_{x_n\leq n_0} \dfrac{\log x_i}{x_i}+\dfrac{1}{t}\sum_{n_0<x_n\leq t}\dfrac{\log x_i}{x_i}<\dfrac{1}{t} \sum_{x_n\leq n_0} \dfrac{\log x_i}{x_i}+\frac{t-n_0}{t}\epsilon.$$
		Thus,
		$$\lim_{t\rightarrow \infty}\dfrac{1}{t}\sum_{i=1}^{t}\dfrac{\log x_i}{x_i} \leq \epsilon.$$
		Since $\epsilon$ is an arbitrary positive real number, we have
		$$\lim_{t\rightarrow \infty}\dfrac{1}{t}\sum_{i=1}^{t}\dfrac{\log x_i}{x_i}=0.$$
	\end{proof}	
	\begin{lemma}\label{lemma-3}
		Let $(a_{i})^{\infty}_{i=1}$ be an increasing sequence real numbers such that  $$\liminf_{n\rightarrow \infty } \dfrac{a_{n+1}}{a_{n}} >1.$$ Then there exists $n_{0} \in \mathbb{N}$ and a real number $\alpha >1$ such that $\dfrac{a_{n+1}}{a_{n}} > \alpha$ for all $n \geq n_{0}$.
	\end{lemma}
	\begin{proof}
		Proof follows from the definition of limit infimum.
	\end{proof}

	\section{Proof of Theorem \ref{Ratio test}}\label{Proof of Theorem}
	The proof of the Theorem \ref{Ratio test} is divided into two auxiliary propositions. 
	\begin{proposition}\label{existence_disjoint_additive_complement}
		Let  $m$ and $n$ be two natural numbers with $n\geq2m$, and let $l>1$ be a real number. Let $A$ be an infinite set of natural numbers such that $|A\cap [1,m]|> |A\cap (m,ln]|$. Then  $$(n,ln]\subset A+ ((m,ln]\setminus A).$$
	\end{proposition}
	\begin{proof}
		Since $|A\cap [1,m]|> |A\cap (m,ln]|$, we have $A\cap [1,m]\neq \varnothing$. Let $u$ be the largest element in the set $A\cap [1,m]$. Using the hypothesis $n\geq 2m$, we get that $(n,ln]\subseteq u+(m,ln].$ This gives 
		$$(n,ln]\subset\left( A+((m,ln]\setminus A)\right)\bigcup\left( u+((m,ln]\cap A)\right),$$
		as $u\in A$. Therefore, to show the conclusion of this proposition, we need to show that
		\begin{equation}
			u+X\subset A+((m,ln]\setminus A), \label{THMlem1eq1}
		\end{equation}
		where $X:=\{x\in (m,ln]\cap A : u+x\in (n,ln] \}.$ 
		
		Now, we shall prove expression \eqref{THMlem1eq1}. If $X = \varnothing$, then we are done. Otherwise, let $y\in X.$ Define
		$$U_y:=\{u+y-a: a\in A\cap [1,m]\}.$$
		Since $|A\cap[1,m]|> |A\cap (m,ln]|$, we have 
		\begin{equation}\label{THMlem1eq2}
			|U_y|> |A\cap (m,ln]|.
		\end{equation}  
		Since  $y\in X$ and $u=\max(A\cap [1,m])$, we get that
		$$m<y\leq y+u-a<y+u \leq ln ~\forall~ a\in A\cap [1,m].$$
		This gives  \begin{equation}
			U_y\subseteq (m, ln].\label{THMlem1eq3}
		\end{equation} 
		By combining  \eqref{THMlem1eq2} and \eqref{THMlem1eq3}, we obtain $$U_y\cap ((m,ln]\setminus A)\neq \varnothing.$$
		Let  $v\in U_y$ such that   $v\in  (m,ln]\setminus A$. Then    $$u+y=b+v\in  A+ (m,ln]\setminus A,$$ for some   $b\in A$.
		This proves the expression \eqref{THMlem1eq1}, which completes the proof of the lemma.
	\end{proof}
	
	\begin{proposition}\label{process_for_refined_additive_complement}
		Let $x_1$, $x_2$, $m$, and $n$ be natural numbers with $m+n\leq x_2$. 	Let  $A$ and  $B$ be non-empty  sets  of natural numbers such that $B \subseteq (x_{1}, x_{2}]$,  and $|A \cap [1,m-x_1)|> x_2-x_1 
		-|B|$. If 
		$(m, m+n]\subseteq \bigcup_{i\in B} (A+ i)$, then there exists a set $S\subseteq B $ such that $(m, m+n]\subseteq \bigcup_{i\in S} (A+ i)$ and 
		$$ |S|\leq (C|B|
		+n) \dfrac{\log\left(|A \cap [1,m-x_1)|- (x_2-x_1 -|B|)\right)}{|A \cap [1,m-x_1)|- (x_2-x_1 -|B|)},$$
		where $C$ is a constant and which is independent from $m$, $n$, $x_1$, $x_2$, $A$, and $B$.
	\end{proposition}
	\begin{corollary}\label{process_for_refined_additive_complement_corollary}
		Let $q$ be a natural number  and $A$  be an infinite set of natural numbers such that $|A\cap[1,q)|>|A\cap (q,4q]|$. If 	
		$(2q,4q]\subset A+ \left((q,4q]\setminus A\right)$, then there exists a set $S\subset\left((q,4q]\setminus A\right) $ such that
		\begin{equation*}
			(2q,4q]\subset A+ S \quad \text{and} \qquad |S|\leq C\dfrac{q\log\left(|A \cap [1,q)|- |A\cap (q,4q]|\right)}{|A \cap [1,q)|- |A\cap (q,4q]|},
		\end{equation*}
		where $C$ is a constant which is independent of  $q$ and $A$. 
	\end{corollary}
	\begin{proof}
		Putting $x_1 = q, x_2 =4q, m=2q,n=2q$ and $B = (q,4q]\setminus A$ in Proposition \ref{process_for_refined_additive_complement}, we get the proof of this corollary.
	\end{proof}
	By combining Proposition \ref{process_for_refined_additive_complement} and Corollary \ref{process_for_refined_additive_complement_corollary}, we obtain the proof of Theorem \ref{Ratio test} in the following way: 
	\begin{proof}[Proof of Theorem \ref{Ratio test}]
		Since  $\liminf_{n\rightarrow \infty } (a_{n+1}/a_{n}) >1$, by Lemma \ref{lemma-3},  there exists a natural  number $n_{0}$ and a real number $\alpha>1$ such that 
		\begin{equation}
			a_{n+1}\geq \alpha a_n ~\forall~ n\geq n_0.\label{MTHeq1}
		\end{equation}
		Since $\alpha>1$, we can choose $r\in \mathbb{N}$ such that \begin{equation}
			\alpha^r\geq 4.\label{MTHeq2}
		\end{equation}
		Let $p$ be a positive integer such that $p>\max\{a_{n_0}, a_{2r+1}\}$.
		
		Let $q$ be a positive integer such that $q\geq p$. Then there exist positive integers $u_q$ and $m_q$  such that $u_q \geq n_0$ and
		\begin{equation}
			A\cap (q, 4q]=\{a_{u_q+1},a_{u_q+2},\ldots ,a_{u_q+m_q}\}.\label{MTHeq3}
		\end{equation}
		This together with \eqref{MTHeq1} and \eqref{MTHeq2} gives
		$$a_{u_q+r+1}\geq \alpha^{r}a_{u_q+1} > 4q.$$ 
		This implies that $a_{u_q+r+1}\not\in A\cap (q,4q]$ and so,
		$|A\cap (q,4q]|=m_q\leq r$. Therefore, we have
		\begin{equation}
			|A\cap (x,4x]|\leq r ~\forall ~ x\geq p.\label{THMeq4}
		\end{equation}
		
		Since $ p>a_{2r+1}$, we have $|A\cap[1,x)|\geq 2r+1$ for every $x\geq p$. This alongwith inequality $\eqref{THMeq4}$, we get that
		\begin{equation}
			|A\cap (x,4x]|<|A\cap[1,x)|~\forall~ x\geq p. \label{THMeq4*}
		\end{equation}
		By Proposition \ref{existence_disjoint_additive_complement} this gives us 		$$(2x,4x]\subset A+ \left((x,4x]\setminus A\right) ~\forall~ x\geq p.$$
		Using  this  alongwith the ineauality \eqref{THMeq4*}   in Corollary \ref{process_for_refined_additive_complement_corollary}, we get that for every $m\geq p$ there exists a set $S_m \subseteq (m,4m]\setminus A $ such that
		\begin{equation}
			(2m,4m]\subset A+ S_m,\label{THMeq5}
		\end{equation} and
		\begin{equation}
			|S_m|\leq C m \dfrac{\log (|A\cap [1,m]|-|A\cap (m,4m]|)}{|A\cap [1,m]|-|A\cap (m,4m]|}. \label{THMeq6}
		\end{equation}
		Further, using 
		\eqref{THMeq4} and the fact that $\dfrac{\log x}{x}$ is a decreasing function   on the  domain $(1,\infty)$, we have 
		\begin{equation}
			|S_m|\leq C m \dfrac{\log (|A\cap [1,m]|-r)}{|A\cap [1,m]|-r} ~\text{for every}~m\geq p. \label{THMeq7}
		\end{equation}  
		Since $p>a_{2r+1}$, we have 
		$|A\cap [1,m]|\geq 2r+1$ for every $m\geq p$. Therefore 
		\begin{equation}
			|S_m|\leq2 C m \dfrac{\log (|A\cap [1,m]|)}{|A\cap [1,m]|} ~\text{for every}~ m\geq p. \label{THMeq8}
		\end{equation}
		Define  $$B:=\bigcup_{i=2+\lfloor\log_2  (p)\rfloor}^{\infty} S_{2^i}.$$
		Observe that, 
		by  taking $m$ as a power of $2$ in the expressions \eqref{THMeq5} and \eqref{THMeq8}, we get that for every positive integer $l>1+\log_2(p)$,
		$S_{2^{l-1}}\subset (2^{l-1},2^{l+1}]\setminus A$ and 
		$(2^l,2^{l+1}] \subset A+ S_{2^{l-1}}$. This proves that $B$ is an additive complement of $A$ in the complement of $A$. 
		
		Now, we claim that the density of $B$ is zero.  To get this, observe that 
		$$B\cap [1,n]\subset \bigcup_{i=\gamma}^{\beta_n} S_{2^i} ~\forall~ n\in \mathbb{N}$$
		where $\gamma=2+\lfloor\log_2  (p)\rfloor$ and $\beta_n=\max\{j:2^j<n\}.$ This gives us
		$$|B\cap [1,n]|\leq \sum_{i=\gamma}^{\beta_n} 2 C 2^i \dfrac{\log (|A\cap [1,2^i]|)}{|A\cap [1,2^i]|}\leq 4C\sum_{i=\gamma}^{\beta_n}   \sum_{j=2^{i-1}+1}^{2^i} \dfrac{\log (|A\cap [1,j]|)}{|A\cap [1,j]|}.$$
		This implies that $$|B\cap [1,n]|\leq 4C \sum_{j=p}^{n} \dfrac{\log (|A\cap [1,j]|)}{|A\cap [1,j]|}.$$
		Thus, by Lemma \ref{density_calculation}, we get that 
		$$\lim_{n\rightarrow\infty}\dfrac{|B\cap [1,n]|}{n}=0.$$
		Therefore, the density of the set  $B$ is zero.
		
	\end{proof}
	\section{Proof of Proposition \ref{process_for_refined_additive_complement}}
	We shall use the following lemmas to prove this proposition.
	\begin{lemma}\label{counting_translation_element}
		Let $A$ be an infinite set of natural numbers and $B$ be a nonempty finite set of natural numbers such that $B \subseteq (a,b]$. For every $n\in \mathbb{N}$, 
		\begin{align*}
			\sum_{i\in B} 1_{A+i}(n)&\geq |A\cap [n-b,n-a)|-|(a,b]\setminus B|.
		\end{align*}

	\end{lemma}
	\begin{proof}
		\begin{align*}
			\sum_{i\in B} 1_{A+i}(n)&=| A \cap (n-B)|\\
			&=|A\cap [n-b,n-a)|-|A\cap (n-((a,b]\setminus B))|\\ &\geq |A\cap [n-b,n-a)|-|(a,b]\setminus B|.
		\end{align*}
		
	\end{proof}
	\begin{lemma}\label{counting_all_translation_element}
		Let $A$ be an infinite set of natural numbers and $B$ be a non-empty finite set of natural numbers. Let $R$ be a non-empty finite set of natural numbers such that every translation of the form $A+x$, where $x\in B$, contains at most $r$ elements of $R$. Then 
		
		$$\left| \sum_{b\in B}\sum_{t\in R} 1_{A+b}(t)\right|\leq r|B|.$$
	\end{lemma}
	\begin{proof}
		Since  every translate of the form $A+x$, where $x\in B$, contain at most $r$ elements of $R$, we have 
		$$\sum_{t\in R} 1_{A+b}(t)\leq r ~\forall~ b\in B.$$
		Therefore,
		$$ \sum_{b\in B}\sum_{t\in R} 1_{A+b}(t)\leq  \sum_{b\in B}r=  r|B|.$$
	\end{proof}
	\begin{proof}[Proof of Proposition \ref{process_for_refined_additive_complement}]
		Define $f: \mathbb{N}\rightarrow \mathcal{P}(\mathbb{N})$ as follows: 
		$$f(x):=(A+x)\cap (m, m+n].$$ 
		It is given that $(m, m+n]\subset \bigcup_{i\in B} (A+ i)$. So,
		by induction,  we choose a sequence $(b_i)_{i=1}^t$ of distinct elements of $B$ such that for every $j\in [1,t]$, we have 
		\begin{align}
			&\left|f(b_j)\setminus \left(\bigcup_{i=1}^{j-1}f(b_i)\right) \right|= \max\left\{ \left|f(r) \setminus \left(\bigcup_{i=1}^{j-1}f(b_i)\right)\right |: r\in B\setminus \bigcup_{i=1}^{j-1}\{b_i\}\right\}, \label{prop2eq0}
		\end{align} 
		where  $t$	 is the smallest positive integer such that  
		\begin{equation*}
			(m,m+n] \subseteq \bigcup_{i=1}^{t} f(b_i).
		\end{equation*} Let $S=\bigcup_{i=1}^t \{b_i\}$. Then $(m,m+n] \subseteq A+S$. Now, we compute the cardinality of the set $S$.
		
		Let $q:=|f(b_1)|$. For every $j \in \mathbb{N}$, define  
		\begin{equation}\label{prop-2-eq-1}
			T_j:=\left\{r\in [1,t]: \left|f(b_r)\setminus \left(\bigcup_{i=1}^{r-1}f(b_i)\right) \right|=j\right\} \quad \text{and}\quad K_j :=|T_j|.
		\end{equation}
		From  \eqref{prop2eq0}, observe that  $T_j=\varnothing$ and $	K_j=0 $ for each integer $j>q$. 
		Also, from  \eqref{prop2eq0} and \eqref{prop-2-eq-1} we obtain  a decreasing  sequence $(n_i)_{i=0}^q$ of positive integers where $n_0=t$ and $n_q=1$ such that
		\begin{equation*}
			T_q=[1,n_{q-1}], T_{q-1}=(n_{q-1},n_{q-2}], \cdots, T_{j}=(n_j,n_{j-1}], \cdots, T_1=(n_{1}, t]. \label{prop2eq1}
		\end{equation*} 
		Therefore,
		\begin{equation}\label{eq-xx}
			|S|=\left|\bigcup_{x=1}^q T_x\right|=\sum_{x=1}^{q}K_x.
		\end{equation}  Thus, computation of $K_x$ will help us to compute  $|S|$.
		
		Define $H: [0,q]\rightarrow \mathcal{P}(\mathbb{N})$ be a function as follows: 
		$$H(x):=(m,m+n]\setminus \left(\bigcup_{i=x+1}^{q}\bigcup_{r\in T_i} f(b_r)\right).$$	
		One can easily see  that  $H(x-1)\subset H(x)$ for $x\geq 1$. This implies that
		\begin{align*}
			|H(x)|-|H(x-1)|&=
			|H(x)\setminus H(x-1)|\\
			&=\left|\left(\bigcup_{i=x}^{q}\bigcup_{r\in T_x} f(b_r)\right)\setminus \left(\bigcup_{i=x+1}^{q}\bigcup_{r\in T_i} f(b_r)\right)\right|\\
			&=\left|\left(\bigcup_{r\in T_x} f(b_r)\right)\setminus \left(\bigcup_{i=x+1}^{q}\bigcup_{r\in T_i} f(b_r)\right)\right|\\
			&=\left|\left(\bigcup_{r=n_x+1}^{n_{x-1}}f(b_r)\right)\setminus \left(\bigcup_{r=1}^{n_x} f(b_r)\right)\right|\\
			&=\left|\left(\bigcup_{r=n_x+1}^{n_{x-1}}(f(b_r)\setminus \cup_{i=n_x+1}^{r-1}f(b_i))\right)\setminus \left(\bigcup_{r=1}^{n_x} f(b_r)\right)\right|\\
			&=\left|\bigcup_{r=n_x+1}^{n_{x-1}}\left(f(b_r)\setminus \bigcup_{i=1}^{r-1}f(b_i)\right)\right|=\left|\bigcup_{r\in T_x}\left( f(b_r)\setminus \left(\bigcup_{i=1}^{r-1} f(b_i)\right)\right)\right|\\
			&=\displaystyle\sum_{r\in T_x}\left|f(b_r)\setminus \left(\bigcup_{i=1}^{r-1} f(b_i)\right)\right|=x|T_x|. 
		\end{align*}
		Therefore,
		\begin{equation}
			|H(x)|-|H(x-1)|=xK_x.\label{prop2eq2}
		\end{equation}
		Observe from this equation that we can get the value of $K_x$ by computing the cardinality of  $H(x)$.
		
		Now, we compute the cardinality of  $H(x)$ by computing the lower and upper bound of  $\sum_{i\in B}\sum_{t\in H(x)} 1_{A+i}(t)$. To get the upper bound of this quantity, 
		we shall count the  number of elements of the set $f(u)$ inside  the set $H(x)$ for every $x\geq 1$ and for every $u\in B$.  Observe that for every $u\in B$ and $x\in [0,q]$, we have
		\begin{align*}
			\left|f(u)\cap H(x)\right|&=\left|f(u)\setminus \left(\bigcup_{i=x+1}^{q}\bigcup_{b\in T_i} f(b)\right)\right|
			=\left|f(u)\setminus \left(\bigcup_{i=1}^{n_x} f(b_i)\right)\right|. 
		\end{align*} 
		Using \eqref{prop2eq0} in the above equation, we obtain 
		\begin{equation*}
			\left|f(u)\cap H(x)\right|\leq \left|f(b_{n_x+1})\setminus \left(\bigcup_{i=1}^{n_x} f(b_i)\right)\right|\leq x, 
		\end{equation*}
		because $n_x+1\in T_j$ for some $j\in [1,x]$. Therefore,
		\begin{equation}
			\left|f(u)\cap H(x)\right|\leq x ~\forall~ u\in B, x\in [0,q]. \label{prop2eq3}
		\end{equation}
		This together with Lemma \ref{counting_all_translation_element}  gives  that
		\begin{equation}
			\sum_{i\in B}\sum_{t\in H(x)} 1_{A+i}(t)\leq x|B|~ \forall~ x\in [0,q].\label{prop2eq4}
		\end{equation}
		We shall use  Lemma \ref{counting_translation_element} to get a lower bound of $\sum_{i\in B}\sum_{t\in H(x)} 1_{A+i}(t)$. Since $B \subseteq (x_{1}, x_{2}]$, Lemma \ref{counting_translation_element} implies that
		\begin{align*}
			\sum_{i\in B}\sum_{t\in H(x)}1_{A+i}(t)&=\sum_{t\in H(x)}\sum_{i\in B} 1_{A+i}(t)\\
			&\geq \sum_{t\in H(x)}\left(|A \cap [t-x_2,t-x_1)|- (x_2-x_1-|B|)\right),
		\end{align*} for every $x \in [0,q]$.
		Note that $t-x_2\leq 0$ for every $t\in H(x)$, because $H(x)\subset (m,m+n]$ and $m+n\leq x_2$. Thus, for every $x\in [0,q]$, we have
		\begin{align}
			\sum_{i\in B}\sum_{t\in H(x)}1_{A+i}(t)&\geq \sum_{t\in H(x)}\left(|A \cap [1,t-x_1)|- (x_2-x_1 -|B|)\right)\nonumber\\
			& \geq |H(x)| \left(|A \cap [1,m-x_1)|- (x_2-x_1 -|B|)\right).\label{prop2eq6}
		\end{align}
		Combining  \eqref{prop2eq4} and \eqref{prop2eq6}, we get that
		\begin{equation*}
			|H(x)| \left(|A \cap [1,m-x_1)|- (x_2-x_1 -|B|)\right)\leq x|B|.
		\end{equation*}
		Therefore, for every $x\in [0,q]$, we obtain \begin{equation}
			|H(x)|\leq \dfrac{x|B|}{|A \cap [1,m-x_1)|- (x_2-x_1-|B|)}.\label{prop2eq9}
		\end{equation}

		Now, we compute the upper bound of the cardinality of the set $S$. Let $q_0$ be an arbitrary positive integer. Then the equation \eqref{eq-xx} implies that

		\begin{equation}
			|S|=\begin{cases}
				\sum_{x=1}^{q_0}K_x+\sum_{x=q_0+1}^{q}K_{x}, & q_0<q;\\
				\sum_{x=1}^{q}K_x,  & q_0\geq q.\label{prop2eq7}
			\end{cases}
		\end{equation}
		Now, we approximate the sum $\sum_{x=q_0+1}^{q}K_{x}$. Since $f(b_i)\subset (m,m+n]$ for each $i\in\cup_{x=1}^{q} T_x$, we get that
		\begin{align*}
			n\geq \left|\bigcup_{x=q_0+1}^{q}\bigcup_{j\in T_x}\left(f(b_{j})\setminus \left(\bigcup_{i=1}^{j-1} f(b_i)\right)\right)\right|&=\sum_{x=q_0+1}^{q}\sum_{j\in T_x}\left|f(b_{j})\setminus \left(\bigcup_{i=1}^{j-1} f(b_i)\right)\right|\\
			&=\sum_{x=q_0+1}^{q}\sum_{j\in T_x}x\\ &=\sum_{x=q_0+1}^{q}xK_x\\
			&>q_0\sum_{x=q_0+1}^{q}K_x.
		\end{align*}
		
		Therefore,
		\begin{equation}
			\sum_{x=q_0+1}^{q}K_{x}<\dfrac{n}{q_0}.\label{prop2eq8}
		\end{equation}
		Observe from the expressions \eqref{prop2eq2}   that 
		
		\begin{align*}
			\sum_{x=1}^{q'}K_{x} &= \sum_{x=1}^{q'}\dfrac{|H(x)|-|H(x-1)|}{x}  
			=  \sum_{x=1}^{q'-1}\dfrac{1}{x(x+1)}|H(x)| + \dfrac{1}{q'}|H(q')| ~\forall~ q'\leq q.
		\end{align*}
		The upper bound of $|H(x)|$ given in  \eqref{prop2eq9} and   the above equation implies that
		\begin{align}
			\sum_{x=1}^{q'}K_{x}& \leq \frac{|B|}{\left(|A \cap [1,m-x_1)|- (x_2-x_1-|B|)\right)}\left(\sum_{x=1}^{q'}\dfrac{1}{x}\right) ~\forall~ q'\leq q. \label{propo2eq11}
		\end{align}
		By combining \eqref{prop2eq7}, \eqref{prop2eq8} and \eqref{propo2eq11}, we obtain
		\begin{equation}
			|S|\leq \frac{|B|}{|A \cap [1,m-x_1)|- (x_2-x_1-|B|)}\left(\sum_{x=1}^{q_0}\dfrac{1}{x}\right)+\frac{n}{q_0}.\label{prop2eq11}
		\end{equation}
		For every $j\in \mathbb{N}$, we have  $\sum_{i=1}^{j}\dfrac{1}{i}<C\log j$, where $C$ is a constant independent to $j$. Therefore,
		\begin{equation}
			|S|\leq \frac{C|B|\log q_0}{|A \cap [1,m-x_1)|- (x_2-x_1-|B|)}+\frac{n}{q_0}.\label{prop2eq12}
		\end{equation}
		Since $q_0$ is an arbitrary positive integer, we  choose  $$q_0=\left\lfloor\dfrac{|A \cap [1,m-x_1)|- (x_2-x_1-|B|)}{\log({|A \cap [1,m-x_1)|- (x_2-x_1-|B|)})}\right\rfloor.$$ Putting this value of $q_0$ in  \eqref{prop2eq12}, we obtain
		\begin{align*}
			|S| & \leq (C|B|
			+n) \dfrac{\log(|A \cap [1,m-x_1)|- (x_2-x_1-|B|))}{|A \cap [1,m-x_1)|- (x_2-x_1-|B|)}.
		\end{align*}

	\end{proof}
	\section{Conclusion}
	
	We proved the ratio test for the existence of the additive complement of a given set $A$ in the complement of the set $A$. This test does not work for those sets $A = \{a_{i} \in \mathbb{N}: i\in \mathbb{N}~ a_{i}< a_{i+1}\}$ such that $\liminf_{n\rightarrow \infty } (a_{n+1}/a_{n}) =1$.  In the case of  $\liminf_{n\rightarrow \infty } (a_{n+1}/a_{n}) =1$, we can observe some examples for non-existence of any additive complement in the complement of the set $A$ using Corollary \ref{non-exitance of additive complement}. If we take a strictly increasing sequence   $(c_n)_{n=1}^{\infty}$ formed by the all composite numbers, then we observe that $\liminf_{n\rightarrow \infty } (c_{n+1}/c_{n}) =1$  and the set of all prime numbers works as an additive complement in the complement of the set of composite numbers.   This motivates us to ask the following question. 
	\begin{question}
		Let $A = \{a_{i} \in \mathbb{N}: i\in \mathbb{N}~ a_{i}< a_{i+1}\}$. Give a sufficient condition on the set $A$ with $\liminf_{n\rightarrow \infty } (a_{n+1}/a_{n}) =1$ such that there exists an additive complement of the set $A$ in the complement of $A$.
	\end{question}
	We are also interested in the following question related to the root test for additive complement.
	\begin{question}
		Let $A = \{a_{i} \in \mathbb{N} : i\in \mathbb{N}\}$ such that $a_{i}<a_{i+1}$ for all $i \in \mathbb{N}$. If $\liminf_{n\rightarrow \infty } (a_{n})^{\frac{1}{n}} >1$, then does there exist a set $B$, a set of natural numbers with density zero, such that $B$ is an additive complement of $A$ in the complement of $A$.
	\end{question}
	
	If we think of a relation between two additive complements of a given set, then we observe that it is easy to find such examples of sets $A$ such that it has at least two disjoint additive complements. So, can we get an example of a set $A$ such that every additive complement intersects other additive complement of the set  $A$?  
	\bibliographystyle{amsplain}

\end{document}